\documentclass[letterpaper, 12pt]{amsart}

\usepackage{amsmath,amssymb,amsthm,amscd,enumerate,setspace}
\usepackage[all]{xy}
\usepackage{graphicx}
\usepackage{xcolor}

\usepackage[pagebackref]{hyperref}
\hypersetup{colorlinks=true, linkcolor=red, citecolor=blue}

\long\def\symbolfootnote[#1]#2{\begingroup%
\def\thefootnote{\fnsymbol{footnote}}\footnote[#1]{#2}\endgroup}

\newtheorem{Theorem}{Theorem}[section]
\newtheorem{Lemma}[Theorem]{Lemma}
\newtheorem{Corollary}[Theorem]{Corollary}
\newtheorem{Proposition}[Theorem]{Proposition}

\newtheorem{Remark}[Theorem]{Remark}

\theoremstyle{definition}
\newtheorem{Example}[Theorem]{Example}
\newtheorem{Definition}[Theorem]{Definition}
\newtheorem{Question}[Theorem]{Question}

\newtheorem{claim}{Claim}

\def\fka{\mathfrak a}

\def\fkm{\mathfrak m}

\def\fkp{\mathfrak p}
\def\fkq{\mathfrak q}

\def\fkP{\mathfrak P}

\def\rmD{\mathrm D}

\def\rmK{\mathrm K}
\def\rmH{\mathrm H}

\newcommand{\calR}{\mathcal{R}}
\def\K{\mathrm{K}}
\def\p{\mathfrak p}
\def\height{\operatorname{ht}}

\newcommand{\rmQ}{\mathrm{Q}}
\def\m{\mathfrak m}

\def\lb{[\![}
\def\rb{]\!]}

\def\Ass{\operatorname{Ass}}
\def\Assh{\operatorname{Assh}}

\def\depth{\operatorname{depth}}
\def\ds{\displaystyle}

\def\Ext{\mbox{\rm Ext}}

\def\h{\operatorname{ht}}
\def\Hom{\operatorname{Hom}}
\def\pd{\operatorname{pd}}

\def\Im{\operatorname{Im}}

\def\ker{\operatorname{Ker}}

\def\Min{\operatorname{Min}}

\def\rmS{\mathrm S}
\def\rmG{\mathrm G}

\def\Spec{\operatorname{Spec}}
\def\Supp{\operatorname{Supp}}

\def\th{\mbox{\tiny th}}

\def\rar{\rightarrow}
\def\lrar{\longrightarrow}
\def\Rar{\Rightarrow}

\def\inj{\hookrightarrow}

\begin{document}

\title{\sc Rings with $q$-torsionfree canonical modules}

\author[N. Endo]{Naoki Endo}
\address{School of Political Science and Economics, Meiji University, 1-9-1 Eifuku, Suginami-ku, Tokyo 168-8555, Japan}
\email{endo@meiji.ac.jp}
\urladdr{https://www.isc.meiji.ac.jp/~endo/}

\author[L. Ghezzi]{Laura Ghezzi}
\address{Department of Mathematics, New York City College of Technology and the Graduate Center, The City University of New York, 300 Jay Street, Brooklyn, NY 11201, U.S.A.; 365 Fifth Avenue, New York, NY 10016, U.S.A.}
\email{lghezzi@citytech.cuny.edu}

\author[S. Goto]{Shiro Goto}
\address{Department of Mathematics, School of Science and Technology, Meiji University, 1-1-1 Higashi-mita, Tama-ku, Kawasaki 214-8571, Japan}
\email{shirogoto@gmail.com}

\author[J. Hong]{Jooyoun Hong}
\address{Department of Mathematics, Southern Connecticut State University, 501 Crescent Street, New Haven, CT 06515-1533, U.S.A.}
\email{hongj2@southernct.edu}

\author[S.-i. Iai]{Shin-ichiro Iai}
\address{Mathematics laboratory, Sapporo College, Hokkaido University of Education, 1-3 Ainosato 5-3, Kita-ku, Sapporo 002-8502, Japan}
\email{iai.shinichiro@s.hokkyodai.ac.jp}

\author[T. Kobayashi]{Toshinori Kobayashi}
\address{Department of Mathematics, School of Science and Technology, Meiji University, 1-1-1 Higashi-mita, Tama-ku, Kawasaki 214-8571, Japan}
\email{toshinorikobayashi@icloud.com}

\author[N. Matsuoka]{Naoyuki Matsuoka}
\address{Department of Mathematics, School of Science and Technology, Meiji University, 1-1-1 Higashi-mita, Tama-ku, Kawasaki 214-8571, Japan}
\email{naomatsu@meiji.ac.jp}

\author[R. Takahashi]{Ryo Takahashi}
\address{Graduate School of Mathematics, Nagoya University, Furocho, Chikusaku, Nagoya, Aichi 464-8602, Japan}
\email{takahashi@math.nagoya-u.ac.jp}

\thanks{N. Endo was partially supported by JSPS Grant-in-Aid for Young Scientists 20K14299. 
L. Ghezzi was partially supported by the Fellowship Leave from the New York City College of Technology-CUNY (Fall 2022-Spring 2023) and by a grant from the City University of New York PSC-CUNY Research Award Program Cycle 53.
S. Goto was partially supported by JSPS Grant-in-Aid for Scientific Research (C) 21K03211. 
J. Hong was partially supported by the Sabbatical Leave Program at Southern Connecticut State University (Spring 2022).
T. Kobayashi was partly supported by JSPS Grant-in-Aid for JSPS Fellows 21J00567.
N. Matsuoka was partially supported by JSPS Grant-in-Aid for Scientific Research (C) 18K03227.
R. Takahashi was partially supported by JSPS Grant-in-Aid for Scientific Research (C) 19K03443.}

\maketitle

\setlength{\baselineskip} {15.3pt}

\begin{center} Dedicated to the memory of Wolmer V. Vasconcelos \end{center}

\begin{abstract}
Let $A$ be a Noetherian local ring with canonical module $\rmK_A$. We characterize $A$ when $\rmK_A$ is a torsionless, reflexive, or $q$-torsionfree module for an integer $q \ge 3$. If $A$ is a Cohen-Macaulay ring, H.-B. Foxby proved in 1974 that the $A$-module $\rmK_A$ is $q$-torsionfree if and only if the ring $A$ is $q$-Gorenstein. With mild assumptions, we provide a generalization of Foxby's result to arbitrary Noetherian local rings admitting the canonical module.  In particular, 
since the reflexivity of the canonical module is closely related to the ring being Gorenstein in low codimension, we also explore quasi-normal rings, introduced by W. V. Vasconcelos. We provide several examples as well.
\end{abstract}


\section{Introduction}
This paper investigates the question of the structure of a Noetherian local ring $A$ if its canonical module $\rmK_A$ is a torsionless, reflexive, or more generally, $q$-torsionfree $A$-module for an integer $q \ge 3$. 
The notion of $q$-torsionfree modules was one of the central contributions of the famous research of M. Auslander and M. Bridger \cite{AB}, which was succeeded by H.-B. Foxby \cite{F} to be a striking study of $q$-Gorenstein rings. 
Among many interesting results, Foxby settled the above question in the case where $A$ is a Cohen-Macaulay ring. More precisely, the $A$-module $\rmK_A$ is $q$-torsionfree if and only if the ring $A$ is $q$-Gorenstein, i.e., $A_{\fkp}$ is a Gorenstein ring for every $\fkp \in \Spec A$ with $\depth A_{\fkp} < q$ (see \cite[Proposition 3.2]{F1}).
It remains unclear what happens if we do not assume the ring $A$ is Cohen-Macaulay. 
The theory of canonical modules nowadays has been developed mainly over Cohen-Macaulay rings in connection with the Gorenstein property; see e.g., \cite{BH, GW, HK, K, St}.  
However, over Noetherian local (not necessarily Cohen-Macaulay) rings, there are also remarkable preceding researches on canonical modules, including the study of their endomorphism algebras; see \cite{A, AG, BS}. 
Therefore, behaviors of canonical modules, even for non-Cohen-Macaulay rings, are interesting and the $q$-torsionfree property is well worth studying. 
The motivation for the present research started with this question that arose while the second and fourth authors were writing the last paper with Vasconcelos concerning (torsionless) canonical modules \cite{BGHV}. 

To explain our results more precisely, let us start from definitions which we will use throughout this paper. For a Noetherian local ring $A$ of dimension $d$ with maximal ideal $\fkm$, a {\it canonical module} $K$ of $A$ is a finitely generated $A$-module satisfying
\[
\widehat{A} \otimes_{A} K \cong \Hom_{\widehat{A}}(\rmH^d_{\widehat{\fkm}}(\widehat{A}), \widehat{E}) 
\]
where $\rmH^d_{\widehat{\fkm}}(\widehat{A})$ denotes the $d^{\th}$ local cohomology module of the $\fkm$-adic completion $\widehat{A}$ of $A$ with respect to its maximal ideal ${\ds \widehat{\fkm}}$ and ${\ds \widehat{E}}$ is the injective hull of the $\widehat{A}$-module $\widehat{A} / \widehat{\fkm}$ (\cite[Definition 5.6]{HK}). 
Equivalently, a finitely generated $A$-module $K$ is a canonical module of $A$ if ${\ds \Hom_{A}(K, E) \cong \rmH^{d}_{\fkm}(A)}$, where $\rmH^{d}_{\fkm}(A)$ is the $d^{\th}$ local cohomology module of $A$ with respect to $\m$ and ${\ds E}$ is the injective hull of $A/\fkm$ (\cite[Definition 12.1.2, Remarks 12.1.3]{BS}). 
The canonical module $\rmK_A$ is uniquely determined up to isomorphisms (\cite[(1.5)]{A}, see also \cite[Lemma 5.8]{HK}) if it exists. Although the existence is not guaranteed even for Cohen-Macaulay local domains, provided $A$ is Cohen-Macaulay, the ring $A$ admits the canonical module if and only if $A$ is a homomorphic image of a Gorenstein ring (\cite{Re, S2}). 
The fundamental theory of canonical modules over Cohen-Macaulay rings was developed in the monumental book \cite{HK} of J. Herzog and E. Kunz. We shall in this paper freely refer to \cite{HK} for basic results on canonical modules (see \cite[Chapter 3]{BH} also).

We now continue to state our setup. Let $R$ be a Noetherian (not necessarily local) ring. For an $R$-module $M$, we have a canonical homomorphism 
$$
\varphi: M \rar M^{**}
$$ 
defined by ${\ds \big[ \varphi(x) \big](f) = f(x) }$ for each ${\ds f \in M^{*} }$ and $x \in M$, where $(-)^{*} =\Hom_{R}(-, R)$ denotes the $R$-dual functor. We say that $M$ is {\em torsionless} (resp. {\em reflexive}) if $\varphi$ is injective (resp. bijective). Torsionless modules are torsionfree, i.e., there is no nonzero torsion elements, and the converse holds if the total ring of fractions $\rmQ(R)$ of $R$ is Gorenstein (\cite[Theorem (A.1)]{V1}). 
Moreover, the $R$-module $M$ is torsionless (resp. reflexive) if and only if $\Ext_R^i(\rmD(M), R) = (0)$ for $i=1$ (resp. $i=1, 2$), where $\rmD(M)$ denotes the Auslander transpose of $M$ (\cite{AB}). From this point of view, Auslander and Bridger introduced a {\it $q$-torsionfree module} $M$ to be $\Ext_R^i(\rmD(M), R) = (0)$ for all $i=1, 2, \ldots, q$. 
In addition, for an integer $n$, we say that
\begin{itemize}
\item $M$ satisfies  $(\rmS_{n})$ if ${\ds \depth M_{\fkp} \geq \min \left\{n, \, \dim R_{\fkp} \right\} }$ for every $\fkp \in \Spec R$,
\item $M$ satisfies $(\widetilde{\rmS}_{n})$ if ${\ds \depth M_{\fkp} \geq \min \left\{n, \, \depth R_{\fkp} \right\} }$ for every $\fkp \in \Spec R$,
\item $R$ satisfies $(\rmG_{n})$ if $R_{\fkp}$ is Gorenstein for every $\fkp \in \Spec R$ with ${\ds \dim R_{\fkp} \leq n}$,
\item $R$ satisfies $(\widetilde{\rmG}_{n})$ if $R_{\fkp}$ is Gorenstein for every  $\fkp \in \Spec R$ with ${\ds \depth R_{\fkp} \leq n}$.
\end{itemize}
The condition $(\rmS_{n})$ is known as Serre's condition. 
A Noetherian ring satisfying $(\widetilde{\rmG}_{n})$ coincides with $(n+1)$-Gorenstein ring in earlier publications such as  \cite{AB, F}. The condition $(\widetilde{\rmG}_{n})$ is equivalent to saying that the ring satisfies both $(\rmS_{n+1})$ and  $(\rmG_{n})$.  

Let us now state our results, explaining how this paper is organized. 
In Section \ref{TCM}, after recalling the necessary definitions and preliminaries, we give a criterion for a Noetherian local ring $A$ to have the torsionless canonical module. We show that the $A$-module $\rmK_A$ is torsionless if and only if $A_{\p}$ is Gorenstein for every $\p \in \Assh A$, where $\Assh A =\{\p \in \Spec A \mid \dim A/\p = \dim A\} = \Ass_A \rmK_A$ (Proposition \ref{prop2}). 
Section \ref{RCM} is devoted to the characterizations of local rings $A$ with reflexive canonical modules. When $\dim A = 1$, this is exactly the case where $A$ is a Gorenstein ring (Proposition \ref{1.6}). We elaborate on the one-dimensional case in Section \ref{one-dimensional}. For the higher dimensional case, the reflexivity of $\rmK_A$ is characterized by the local ring $A_{\p}$ being Gorenstein for every $\p \in \Supp_A K_A$ with $\dim A_\p \le 1$ and $\Ass A \cap V(U) = \Assh A$, where $U$ denotes the unmixed component of $(0)$ in $A$ and $V(U)$ is the set of all prime ideals in $A$ containing $U$ (Theorem \ref{refcano1}). This lead us to obtain Corollary \ref{refcano2}, which claims that $\rmK_A$ is reflexive if and only if $A$ satisfies $({\rmG}_{1})$, provided $\Assh A = \Ass A$. This indicates that the reflexivity of canonical modules is deeply related to the ring being Gorenstein in low codimension. Thus Section \ref{QNR} is dedicated to quasi-normal rings, i.e., rings with $(\rmS_{2})$ and $(\rmG_{1})$, which have been introduced by Vasconcelos.
In Section \ref{qTFCM}, we generalize Foxby's result on $q$-torsionfree canonical modules to arbitrary Noetherian local rings $A$ admitting a canonical module. Our results of Sections \ref{TCM} and \ref{RCM} provide a complete generalization in case $q=1,2$. When $q \ge 3$, Theorem \ref{p3.2} states that the $A$-module $\rmK_A$ is $q$-torsionfree if and only if the ring $A$ satisfies $(\rmG_{q-1})$ and $(\rmS_{q-1})$ on $\Supp_A \rmK_A$, provided that $\rmK_A$ satisfies $(\rmS_q)$. In the final section we provide concrete examples of Cohen-Macaulay and $q$-Gorenstein rings in order to illustrate our theorems.



\section{Torsionless canonical modules}\label{TCM}

Throughout the section, let $(A, \fkm)$ be a Noetherian local ring of dimension $d$. We begin with some preliminaries. Let ${\ds (0) = \bigcap_{\fkp \in \Ass A } Q(\fkp) }$ denote a primary decomposition of $(0)$ in  $A$. We set
\[  \Assh A= \{ \fkp \in \Spec A \mid \dim A/\fkp =d \} \quad \mbox{and} \quad  U= \bigcap_{\fkp \in \Assh A} Q(\fkp) \]
where $U$ is called  the {\it unmixed component} of $(0)$ in $A$. Let  $V(U)$ denote the set of all prime ideals of $A$ containing $U$.

\begin{Lemma}\label{lem1}
There is an embedding ${\ds 0 \rar A/U \rar A}$ of $A$-modules.
\end{Lemma}

\begin{proof}
We may assume that $U \neq (0)$. Then $\Assh A \subsetneq \Ass A$. Let 
\[  L = \bigcap_{\tiny \fkp \in \Ass A \setminus \Assh A}Q(\fkp).\]  We then have ${\ds L \not\subseteq \bigcup_{\tiny \fkp \in \Assh A}\fkp}$. Choose an element $a \in L$ but ${\ds a \not\in \bigcup_{\fkp \in \Assh A}\fkp}$. Since $a$ is a non-zerodivisor on $A/U$ and $aU \subseteq L \cap U=(0)$, we obtain $((0):_{A} a) = U$.  Then $U$ is the kernel of the homomorphism $ \varphi : A \rar A$ given by $\varphi(1)=a$.   Thus, $A/U \cong \Im(\varphi) \inj A$.
\end{proof}

In the rest of this section, we assume the ring $A$ admits the canonical module $\rmK_A$. 
We recall several known facts about $\rmK_A$ which we will use throughout this article; see \cite[(1.6), (1.7), (1.8), (1.9), (1.10), Theorem 3.2]{A} and \cite[Korollar 6.3]{HK} (also \cite[Chapter 12]{BS}) for the proofs.

\begin{Proposition}\label{P1}
The following assertions hold true. 
\begin{enumerate}[$(1)$]
\item The annihilator of $\rmK_{A}$ is $U$. In particular, ${\dim_A \rmK_{A}=d}$ and ${\Ass_A \rmK_{A} = \Assh A}$.

\item If ${\ds a \in \fkm}$ is $A$-regular, then $a$ is $\rmK_{A}$-regular.  

\item $V(U) = \Supp_A \rmK_{A} = \{\fkp \in \Spec A \mid \dim A = \dim A/\fkp + \h_A \fkp \}$.

\item Both $\rmK_{A}$ and ${\ds \Hom_{A}(\rmK_{A}, \rmK_{A}) }$ satisfy  $(\rmS_{2})$.

\item ${\ds \rmK_{A_{\fkp}}=\big[\rmK_{A} \big]_{\fkp} }$ for every ${\ds \fkp \in \Supp_A \rmK_{A} }$.

\item $\Supp_A \rmK_{A} = \Spec A$ if and only if $\Min A = \Assh A$. 

\item $\Ass A = \Assh A$ if and only if $\rmK_{A}$ is a faithful $A$-module.

\item Suppose $A$ is Cohen-Macaulay. If $a \in \fkm$ is $A$-regular, then $\rmK_{A/(a)}$ exists and $\rmK_{A/(a)}  \cong \rmK_{A}/a\rmK_{A}$. 

\item Suppose that $\rmK_{A/U}$ exists. Then $\rmK_{A/U}$, as an $A$-module, is the canonical module of $A$.
\end{enumerate}
\end{Proposition}


\medskip


By Proposition \ref{P1}-(2), the canonical module $\rmK_A$ is torsionfree as an $A$-module. In general, torsionless modules are torsionfree, and the converse holds if and only if $A_\p$ is a Gorenstein local ring for every $\p \in \Ass A$ (\cite[Theorem (A.1)]{V1}). Therefore, if the total ring of fractions $\rmQ(A)$ of $A$ is Gorenstein, then $\rmK_A$ is torsionless. The following proposition shows the case where $\rmK_{A}$ is torsionless without assuming $\rmQ(A)$ is Gorenstein. It is also a generalization of \cite[Proposition 3.2]{BGHV}.


\begin{Proposition}\label{prop2}
The following conditions are equivalent{\rm \,:}
\begin{enumerate}[{\rm (1)}]
\item $\rmK_{A}$ is a torsionless $A$-module{\rm \,;}
\item $A_{\fkp}$ is a Gorenstein ring for every $\fkp \in \Assh A${\rm \,;}
\item $\rmK_{A} \cong I$ for some ideal $I$ of $A$.
\end{enumerate}
\end{Proposition}

\begin{proof}
(1) $\Rar$ (2) Since $\rmK_{A}$ is torsionless, there exists an exact sequence $0 \rar \rmK_{A} \rar F$ of $A$-modules, where $F$ is a finitely generated free $A$-module. 
Let ${\fkp \in \Assh A \subseteq \Supp_A \rmK_{A}}$. Then $A_{\fkp}$ is Artinian and $\big[\rmK_A \big]_\fkp$ is the canonical module of  $A_{\fkp}$. 
Therefore, we may assume that $\big[\rmK_A\big]_\fkp$ is the injective hull of $A_\fkp/\fkp A_\fkp$. The splitting monomorphism $0 \rar \big[\rmK_A\big]_{\fkp} \rar F_{\fkp}$ induces that $\big[\rmK_A \big]_{\fkp}$ is  a direct summand of the free $A_\fkp$-module $F_{\fkp}$. Since $A_{\fkp}$ is Artinian, by the Matlis duality, we have $\big[\rmK_A \big]_\fkp \cong A_{\fkp}$. Hence $A_{\fkp}$ is a Gorenstein ring.


\medskip

\noindent 
(2) $\Rar$ (3) Let ${W = A \setminus \bigcup_{\fkp \in \Assh A}\fkp}$. 
By assumption, $\big[\rmK_A \big]_\fkp \cong A_{\fkp}$ for every $\fkp \in \Assh A$. Thus, ${\ds   W^{-1}\rmK_{A} \cong W^{-1}A  }$.  Moreover, we have  ${\ds W^{-1}A \cong W^{-1}(A/U) }$ because  $W^{-1}U=(0)$ by the proof of Lemma \ref{lem1}. Since every element of $W$ is a non-zerodivisor on both $\rmK_{A}$ and $A/U$, the isomorphism ${\ds  W^{-1}\rmK_{A} \cong W^{-1}(A/U) }$  induces the embedding ${\ds \rmK_{A} \inj A/U}$. By Lemma \ref{lem1}, there is an embedding ${\ds \rmK_{A} \inj A}$. 

\medskip 

\noindent (3) $\Rar$ (1) is clear.
\end{proof}

If $A$ is {\it reduced}, which means there are no nonzero nilpotents, then the local ring $A_\p$ is a field for every $\p \in \Ass A$. Hence we obtain the following.

\begin{Corollary}\label{cor3.1} 
 If $A$ is a reduced ring, then $\rmK_{A} \cong  I$ for some ideal $I$ of $A$.
\end{Corollary}

We recall that if $A$ is Cohen-Macaulay, the canonical module $\rmK_A$ has rank one if and only if the ring $A$ is {\it generically Gorenstein}, i.e., $A_{\p}$ is a Gorenstein local ring for every $\p \in \Min A$. When one of the equivalent conditions of  \cite[Proposition 3.3.18]{BH} is satisfied, the canonical module can be identified with an ideal of $A$ (see also \cite[Satz 6.21]{HK}). 
The example below shows that
the assumption $\fkp \in \Assh A$ is necessary for Proposition \ref{prop2}.

\begin{Example}\label{ex1}
Let $S=k\lb X,Y,Z \rb$ be the formal power series ring over a field $k$ and set $A= S/[(X) \cap (Y,Z)^2]$. Let  $x, y, z$ denote the images of $X, Y, Z$ in $A$, respectively. Then we have 
\begin{center}
$U=(x)$ and $\rmK_A  \cong A/(x) \cong y^2A$. 
\end{center}
By Proposition \ref{prop2}, $\rmK_{A}$ is torsionless. However, $A$ is 
not generically Gorenstein. In fact, $A_{\fkq}$ is not a Gorenstein ring for $\fkq = (y,z) \in \Min A$.  
\end{Example}



\section{Reflexive canonical modules}\label{RCM}

Let $(A, \fkm)$ be a Noetherian local ring of dimension $d$ admitting the canonical module $\rmK_{A}$. We denote by $U$ the unmixed component of $(0)$ in $A$.
In this section we will show how the reflexivity of the canonical module is related to the Gorensteinness of the ring.  We begin with the following simple but effective lemma.

\begin{Lemma}\label{1.5}
Suppose that $\rmK_A$ is reflexive. Then $\Ass A \cap V(U) = \Assh A$. 
In particular, if $\rmK_{A}$ is reflexive and $\depth A =0$, then $\dim A =0$. 
\end{Lemma}

\begin{proof} The assertion follows from
\[V(U) \cap \Ass A =  \Supp_A \rmK_{A}^{*} \cap \Ass A = \Ass_A  \Hom_{A}( \rmK_{A}^{*}, A )  = \Ass_A\rmK_{A}^{**}  = \Ass_A\rmK_{A} = \Assh A. \qedhere \]
\end{proof}

The example below shows that the reflexivity of $\rmK_A$ may require a rather strong restriction on $A$.

\begin{Example}\label{ex2}
Let $S=k\lb X,Y \rb$ be the formal power series ring over a field $k$ and set $A= S/[(X) \cap (X^{2}, Y)]$. Let  $x, y$ denote the images of $X, Y$ in $A$, respectively. Let $\fkm=(x, y)$ be the maximal ideal of $A$. Then we have ${\ds \Assh A=\left\{ (x) \right\} }$, ${\ds U=(x) }$, and ${\ds \rmK_{A}=A/U}$.
\begin{enumerate}[(1)]
\item Let $\fkp=(x)$. Since $A_{\fkp}$ is a field, by Proposition~\ref{prop2}, $\rmK_{A}$ is torsionless.

\item Since $\depth A=0$ and $\dim A=1$, by  Lemma~\ref{1.5}, $\rmK_{A}$ is not reflexive.
\end{enumerate}
\end{Example}


\begin{Proposition}\label{1.6}
Suppose $d=1$.
Then $\rmK_{A}$ is a reflexive $A$-module if and only if $A$ is a Gorenstein ring.
\end{Proposition}

\begin{proof}
 Suppose that $\rmK_{A}$ is reflexive. By Lemma \ref{1.5}, $A$ is Cohen-Macaulay.
Since $\rmK_{A}$ is reflexive, there exists an exact sequence 
${\ds 0 \rar \rmK_A \rar F_{1} \rar F_{0} }$, where $F_{0}, F_{1}$ are finite free  $A$-modules \cite[Proposition 2.1]{F}.
Let $a \in \fkm$ be an $A$-regular element. Since $A$ is Cohen-Macaulay,  $\rmK_{A}/a\rmK_{A}$ is the canonical module of $A/aA$. Moreover, the embedding  ${\ds  0 \rar \rmK_A/a\rmK_A \rar F_{1}/aF_{1} }$  proves that $K_{A/aA}$ is torsionless. Therefore, by Proposition \ref{prop2}, $A/aA$ is a Gorenstein ring.  Thus, $A$ is a Gorenstein ring. The converse is clear.
\end{proof}

\begin{Remark}
{\rm There exist non-Cohen-Macaulay local rings with reflexive canonical module. Example \ref{ex3} shows a two-dimensional non-Cohen-Macaulay local ring $A$ with $\rmK_{A}$ reflexive. The example also shows that, even if $\rmK_{A}$ is reflexive, the equality $\Ass A = \Assh A$ does not hold true in general.}
\end{Remark}

Recall that a finitely generated $A$-module $M$ is reflexive, i.e., the canonical map $\varphi : M \to M^{**}$ is an isomorphism, if and only if there is at least one isomorphism $M \cong M^{**}$ of $A$-modules. 

\begin{Lemma}\label{lem3.5}
Suppose that there is an exact sequence
\[ 0 \to \rmK_A \to \rmK_A^{**} \to C \to 0 \]
of $A$-modules. If $C \ne (0)$, then $A_{\p}$ is a Cohen-Macaulay ring with $\dim A_{\p} = 1$ for every $\p \in \Ass_A C$ with $\depth A_{\p} \ge 1$. In particular, $\p \in V(U)$ and $U_{\p} = (0)$.
\end{Lemma}

\begin{proof}
Let $\p \in \Ass_A C$ such that $\depth A_{\p} \ge 1$. Then $\p \in \Supp_A\rmK_A = V(U)$. 
Since $\big[\rmK_A\big]_\p \cong \rmK_{A_{\p}}$, by passing to the ring $A_\p$, we may assume $\depth A>0$ and $\depth_A C=0$. Since $\Ass_A\rmK_A^{**} \subseteq \Ass A$, we have
$ \depth_A \rmK_A^{**} \geq 1$. From the exact sequence  ${\ds 0 \to \rmK_A \to \rmK_A^{**} \to C \to 0}$, we obtain 
\[ 0 = \depth_{A} C \geq \min \{ \depth_{A} \rmK_{A} -1, \; \depth_{A} \rmK_{A}^{**} \}. \]
Thus, $\depth_A \rmK_A=1$. Since $\rmK_A$ satisfies $(\rmS_2)$, we have
\[ 1 = \depth_{A} \rmK_{A} \geq \min \{ 2, \; \dim A \}. \]
Therefore, $A$ is a Cohen-Macaulay ring of dimension $1$.  In particular, $U=(0)$. 
\end{proof}



Now we aim to generalize  Proposition \ref{1.6}. 

\begin{Theorem}\label{refcano1}
The following conditions are equivalent{\rm \,:}
\begin{enumerate}[{\rm (1)}]
\item $\rmK_A$ is a reflexive $A$-module{\rm \,;}
\item $\Ass A \cap V(U) = \Assh A$, and $A_\fkp$ is Gorenstein for every $\fkp \in \Supp_A \rmK_A$ with $\h_A \fkp \le 1$. 
\end{enumerate}
\end{Theorem}

\begin{proof}
$(1) \Rightarrow (2)$ By Lemma \ref{1.5}, we have $\Ass A \cap V(U) = \Assh A$.  Let $\fkp \in \Supp_A \rmK_A$ with $\h_A \fkp \le 1$.  If $\fkp \in \Assh A$, then $A_{\fkp}$ is Gorenstein by Proposition~\ref{prop2}. Otherwise, we have $\dim A_{\fkp} =1$. Since  $\rmK_{A_\fkp}$ is a reflexive $A_\fkp$-module, the ring $A_\fkp$ is Gorenstein by Proposition~\ref{1.6}.

\medskip

\noindent $(2) \Rightarrow (1)$ Since $A_\fkp$ is Gorenstein for every $\fkp \in \Assh A$, by Proposition~\ref{prop2}, $\rmK_A$ is torsionless. Hence we have the exact sequence 
\[0 \rar \rmK_{A} \stackrel{\varphi}{\longrightarrow} \rmK_{A}^{**} \rar C \rar 0 \]
of $A$-modules, where $\varphi$ is the canonical homomorphism. 
Suppose that $C \neq (0)$.  Let $\fkp \in \Ass_A C$. 
Note that $\fkp \in \Supp_A \rmK_{A}$ and $\big[\rmK_{A} \big]_{\fkp} \cong \rmK_{A_{\fkp}}$. 
If $\h_A \fkp \le 1$, then by assumption $A_{\fkp}$ is Gorenstein. Then  $C_{\fkp} =(0)$, which is a contradiction. 
Thus, $\h_A \fkp \ge 2$. 
Since $\Ass A \cap V(U) = \Assh A$, we have $\depth A_\fkp \ge 1$. This shows, by Lemma \ref{lem3.5}, that $A_{\p}$ is a Cohen-Macaulay ring with $\dim A_{\p} = 1$, which is a contradiction. Therefore $C=(0)$ and $\rmK_A$ is a reflexive $A$-module.
\end{proof}


We summarize some consequences of Theorem \ref{refcano1}.
Note that $A$ satisfies $(\rmS_{1})$ if and only if $\Ass A=\Min A$, and the latter condition implies $\Ass A \cap V(U) = \Assh A$. 

\begin{Corollary}\label{refcanom}
If $A$ satisfies  $(\rmS_{1})$ and $(\rmG_{1})$, then $\rmK_{A}$ is a reflexive $A$-module.
\end{Corollary}

Recall that if $\Assh A = \Ass A$, then ${\ds \Spec A = \Supp_{A} \rmK_{A} }$. Thus, we obtain the following as another direct consequence of Theorem \ref{refcano1}.

\begin{Corollary}\label{refcano2}
Suppose that $\Assh A = \Ass A$. 
Then $\rmK_{A}$ is a reflexive $A$-module if and only if $A$ satisfies $(\rmG_{1})$.
\end{Corollary}

If $A$ is a Cohen-Macaulay ring, the above corollary recovers \cite[Korollar 7.29]{HK}. 
Recall that $A$ is a {\it generalized Cohen-Macaulay ring}, if the $i^{\th}$ local cohomology module $\rmH_\fkm^i(A)$ is a finitely generated $A$-module for every $i \neq d$.

\begin{Corollary}
Suppose that $A$ is a generalized Cohen-Macaulay ring and $d >0$.
Then $\rmK_{A}$ is a reflexive $A$-module if and only if 
$\depth A >0$ and $A$ satisfies $(\rmG_{1})$.  
\end{Corollary}

\begin{proof} By assumption, we have ${\ds \Ass A \setminus \{\fkm \} \subseteq \Assh A }$. If $\rmK_A$ is reflexive, then by Lemma \ref{1.5} we see that $\depth A >0$. Without loss of generality, we may assume $\depth A > 0$. Hence ${\ds \Ass A = \Assh A}$. The assertion follows from Corollary \ref{refcano2}.
\end{proof}

It seems natural to ask for the relation between the reflexivity of the $A$-module $\rmK_A$ and that of the $A/U$-module $\rmK_{A/U}$. As for this question, we have
the following.

\begin{Theorem}\label{refcano1s}
The following conditions are equivalent{\rm \,:}
\begin{enumerate}[{\rm (1)}]
\item $\rmK_A$ is a reflexive $A$-module{\rm \,;}
\item $\rmK_{A/U}$ is a reflexive $A/U$-module and $\Ass A \cap V(U) = \Assh A$.
\end{enumerate}
\end{Theorem}

\begin{proof} Let $B=A/U$. Then ${\ds \rmK_{A}=\rmK_{B}}$ (\cite[1.8]{A}). Also note that $\Ass B= \Assh B$.

\medskip

\noindent (1) $\Rightarrow$ (2)
By Lemma \ref{1.5}, we have $\Ass A \cap V(U) = \Assh A$. By Corollary \ref{refcano2},  it suffices to show that $B$ satisfies $(\rmG_{1})$.  Let ${\ds \fkP \in \Spec B}$ be a prime with ${\ds \h_{B} \fkP \leq 1}$. We write  ${\ds \fkP = \fkp/U}$ for some $\fkp \in V(U)$. Then $\h_{A}\fkp = \h_{B} \fkP \leq 1$. Moreover, $\rmK_{A_\fkp}$ is a reflexive $A_\fkp$-module. By Propositions \ref{prop2} and \ref{1.6}, $A_\fkp$ is Gorenstein.
Since ${\ds U_\fkp = (0) :_{A_\fkp} \rmK_{A_\fkp} = (0)}$, we obtain ${\ds B_{\fkP} =A_{\fkp}}$. Thus, ${\ds B_{\fkP}}$ is a Gorenstein ring. 

\medskip

\noindent $(2) \Rightarrow (1)$ Let $\fkp \in \Supp_A \rmK_A$ with $\h_A \fkp \leq 1$. By Theorem \ref{refcano1}, it is enough to show that $A_\fkp$ is Gorenstein. Let $\fkP=\fkp/U$. Then by Corollary \ref{refcano2}, $B_{\fkP}$ is a Gorenstein ring.  Since $\Ass A \cap V(U) = \Assh A$, the ring $A_\p$ is Cohen-Macaulay. In particular, $U_\fkp = (0)$ and ${\ds A_{\fkp} = B_{\fkP}}$.  Therefore $A_\fkp$ is Gorenstein.
\end{proof}

\begin{Corollary}\label{AmodU}
Suppose that $A/U$ is a Gorenstein ring. Then the following assertions hold true. 
\begin{enumerate}[{\rm (1)}]
\item $\rmK_{A}$ is a reflexive $A$-module if and only if ${\ds \Ass A \cap V(U) =\Assh A}$.
\item If $A$ satisfies $(\rmS_1)$, then  $\rmK_{A}$ is reflexive.
\end{enumerate}
\end{Corollary}

\begin{proof} Note that (1) follows directly from Theorem \textcolor{purple}{\ref{refcano1s}}. To prove (2), it is enough to show that ${\ds \Ass A \cap V(U) \subseteq \Assh A}$. Let ${\ds \fkp \in  \Ass A \cap V(U) }$. 
Since $A$ satisfies $(\rmS_1)$, we have $\h_A \fkp =0$. Since $\fkp \in V(U)$, we have ${\ds \dim A = \dim A/\fkp + \h_A \fkp = \dim A/\fkp}$. Therefore $\fkp \in \Assh A$.
\end{proof}

Closing this section, we provide the examples of (not necessarily Cohen-Macaulay) local rings admitting reflexive canonical modules.

\begin{Example}
Let $S = k\lb X, Y_1, Y_2, \ldots, Y_n \rb$ ~$(n \geq 2)$ be the formal power series ring over a field $k$ and let $A= S/[(X^m) \cap J]$ where $m \geq 1$ and $J$ is a $(Y_1, Y_2, \ldots, Y_n)$-primary ideal of $S$. 
Let  $x$ denote the image of $X$ in $A$. Then $U=(x^m)$,  $\Assh A = \{(x)\}$, and $A/U$ is a Gorenstein ring. By Corollary \ref{AmodU}, the $A$-module $\rmK_{A}$ is reflexive. 
\end{Example}

\begin{Example} Let $k$ be a field and $R=k[\Delta]$ be the Stanley-Reisner ring of a simplicial complex $\Delta$. Since $R$ is reduced, the graded canonical module $\rmK_R$ (see \cite{GW, St2}) is torsionless. Moreover, if $\#(\Assh R)=1$, the ring $R/U$ is Gorenstein, so that $\rmK_R$ is reflexive as an $R$-module, where $U$ stands for the unmixed component of $(0)$ in $R$ and $\Assh R = \{\fkp \in \Spec R \mid \dim R/\fkp = \dim R \}$.
\end{Example}



\section{Quasi-normal rings}\label{QNR}


Quasi-normal rings were introduced by Vasconcelos (\cite[Definition 1.2]{V1}) and they are exactly $2$-Gorenstein rings.  


\begin{Definition}\label{QN1}{\rm 
A Noetherian ring $R$ is said to be {\em quasi-normal} if $R$ satisfies $(\rmS_{2})$ and $(\rmG_{1})$. 
}\end{Definition}

The following is a direct consequence of \cite[Proposition 2.3]{F} (for a local ring case, see also \cite[Theorem 3.8]{EG85}).  Here we include our alternative proof specifically for quasi-normal rings.

\begin{Proposition}\label{QN2}
Let $R$ be a quasi-normal ring and let $M$ be a finitely generated $R$-module. 
If ${\ds   \depth_{R_{\fkp}} M_{\fkp} \geq \min \{ 2, \; \dim R_{\fkp} \}  }$ for every $\fkp \in \Spec R$, then $M$ is reflexive.
\end{Proposition}

\begin{proof} Consider the exact sequence of $R$-modules
\[0 \rar X \rar  M \stackrel{\varphi}{\longrightarrow}  M^{**} \rar C \rar 0, \]
where $\varphi$ denotes the canonical homomorphism.
Suppose $X \neq (0)$ and choose $\fkp \in \Ass_R X$.  By assumption, we have $\dim R_{\fkp}=0$. Since $R$ satisfies $(\rmG_{1})$, the local ring ${\ds R_{\fkp}}$ is  Gorenstein, whence $M_{\fkp}$ is reflexive. Hence $X_{\fkp}=(0)$, which is a contradiction. So $X=(0)$, and we have the exact sequence 
\[0 \rar   M \rar  M^{**} \rar C \rar 0. \]
Suppose $C \neq (0)$. Let ${\ds \fkp \in \Ass_R C}$. If $\dim R_{\fkp}=0$, then $M_{\fkp}$ is reflexive, so  $C_{\fkp}=(0)$. This is a contradiction. Thus ${\ds \dim R_{\fkp} \geq 1}$. As $R$ satisfies $(\rmS_{2})$, we have ${\ds \depth R_{\fkp} \geq \min \{ 2, \; \dim R_{\fkp} \} }$. Hence ${\ds \depth R_{\fkp} \geq 1}$. Since 
${\ds \depth_{R_{\fkp}} M_{\fkp}^{**} \geq \min \{2, \; \depth R_{\fkp} \}}$, we then have ${\ds \depth_{R_{\fkp}} M_{\fkp}^{**} \geq 1}$. The exact sequence
\[0 \rar  M_{\fkp} \rar  M_{\fkp}^{**} \rar C_{\fkp} \rar 0 \]
gives that 
\[ 0 = \depth_{R_{\fkp}} C_{\fkp} \geq \min\{ \depth_{R_{\fkp}} M_{\fkp} -1, \;  \depth_{R_{\fkp}} M_{\fkp}^{**} \}. \]
Hence ${\ds \depth_{R_{\fkp}} M_{\fkp}   \leq 1}$. By assumption, we have
\[ 1 \geq  \depth_{R_{\fkp}} M_{\fkp} \geq \min \{ 2, \; \dim R_{\fkp} \}. \]
Therefore ${\ds \dim R_{\fkp} =1}$ and ${\ds \depth_{R_{\fkp}} M_{\fkp}  =1}$. Since $R$ satisfies $(\rmG_{1})$, the ring ${\ds R_{\fkp}}$ is Gorenstein. By \cite[Corollary 2.3]{V1}, we see that $M_{\fkp}$ is reflexive. Hence $C_{\fkp} =(0)$, which is a contradiction.
\end{proof}


 A finitely generated $R$-module $\omega_{R}$ is a {\em canonical module} of $R$, if ${\ds (\omega_{R})_{\fkm}}$ is the canonical module of $R_{\fkm}$ for all maximal ideals $\fkm$ of $R$. In contrast to the local case, the canonical module is in general not unique up to isomorphisms; see e.g., \cite[Remark 3.3.17]{BH}.

\begin{Corollary}\label{QN3}
Let $R$ be a Noetherian ring with $d = \dim R >0$. Suppose that there exists a canonical module $\omega_{R}$ and ${\ds \Ass R_{\fkm}=\Assh R_{\fkm}}$ for every maximal ideal $\fkm$. Then $R$ is quasi-normal if and only if $R$ satisfies $(\rmS_{2})$ and $\omega_{R}$ is reflexive.
\end{Corollary}

\begin{proof} Suppose $R$ is quasi-normal. Since $\omega_{R}$ satisfies $(\rmS_{2})$ and ${\ds \dim_{R_{\fkp}} [\omega_{R}]_{\fkp} = \dim R_{\fkp} }$ for every $\fkp \in \Spec R$, by Proposition \ref{QN2}, we conclude that $\omega_{R}$ is reflexive. For the converse, it remains  to show that $R$ satisfies $(\rmG_{1})$. 
Let $A=R_{\fkm}$, where $\fkm$ is a maximal ideal of $R$. Then $\rmK_{A}= (\omega_{R})_{\fkm}$ is reflexive. Therefore we have ${\ds \Ass A =\Assh A}$. By 
Corollary \ref{refcano2}, $A$ satisfies $(\rmG_{1})$. Thus $R$ satisfies $(\rmG_{1})$.
\end{proof}

We summarize some examples. First, we note examples of quasi-normal rings which are not normal. The simplest ones are non-normal Gorenstein rings.

\begin{Example}\label{4.3}
Suppose that $R$ is a Cohen-Macaulay ring with canonical module $\omega_{R}$. We set $T = R \ltimes \omega_{R}$ to be the idealization of $\omega_R$ over $R$. Then $T$ is a Gorenstein ring (\cite{Re}), but not normal because it is never a  reduced ring.
\end{Example}

For a commutative ring $R$, we denote by $\overline{R}$ the integral closure of $R$ in $\rmQ(R)$. We refer to \cite[p. 178]{BH} for background on numerical semigroups.

\begin{Example}
Let $H=\left<a_1, a_2, \ldots, a_\ell \right>$ be a symmetric numerical semigroup. We consider $R = k[s, t^{a_1}, t^{a_2},\ldots, t^{a_\ell}]$, where $s,t$ are indeterminates and $k$ is a field. Then $R$ is a two-dimensional Gorenstein ring with $\overline{R}=k[s,t]$, so that $R$ is not normal if $1 \not\in H$. As a special case, the ring $R = k[s,t^2,t^3]$ is quasi-normal, but not normal. 
\end{Example}


Next, we note examples of quasi-normal but non-normal Cohen-Macaulay rings which are moreover not Gorenstein.

\begin{Example}
{\rm Let $k$ be a field and $k[X,Y]$ the polynomial ring over $k$. Let $H=\left<a_1, a_2, \ldots, a_\ell \right>$ be a symmetric numerical semigroup such that $1 \not\in H$ and let $k[H]=k[t^{a_1}, t^{a_2}, \ldots, t^{a_\ell}]$ denote the semigroup ring of $H$ over $k$, where $t$ is an indeterminate. Let $T=k[X^n, X^{n-1}Y, \ldots, XY^{n-1},Y^n]$, where $n \ge 3$ is an integer. We set $R= T\otimes_k k[H]$. Then $R$ is a quasi-normal Cohen-Macaulay ring with $\dim R=3$, which is neither Gorenstein nor normal.  Indeed, because $\overline{R}=T \otimes_k k[t]$ and $k[H] \ne k[t]$, the ring $R$ is not normal. As $T$ is normal, we see that $R$ is a quasi-normal ring (see Proposition \ref{5.4} (2)). Moreover, $R$ is not a Gorenstein ring because $T$ is not Gorenstein. The simplest example in this class is $R=k[X^3,X^2Y, XY^2, Y^3]\otimes_k k[t^2,t^3]$.}
\end{Example}

\begin{Example}\label{4.5}
Let $T = k[X,Y,Z,V]$ be the polynomial ring over a field $k$. 
We denote by  $\mathbf{I}_2(\mathbb{N})$ the ideal of $T$ generated by all the $2\times 2$ minors of a matrix $\mathbb{N}$. 
Let $I=\mathbf{I}_2(\mathbb{M})$ where $
\mathbb{M}=\left(\begin{smallmatrix}
X^a&Y^b + V&Z^c\\
Y^{b'}&Z^{c'}&X^{a'}\\
\end{smallmatrix}\right)
$ for some integers $a,b,c,a',b',c' \ge 1$. 
We set $R=T/I$. Then $R$ is a Cohen-Macaulay ring of dimension $2$. 
Let $x,y,z,v$ denote the images of $X,Y,Z,V$ in $R$, respectively. We first check the isomorphism $ \omega_{R} \cong (x^a,y^{b'})R$. In fact, by setting
$$
f = Z^{c+c'}-X^{a'}(Y^b+V), \ \ g = X^{a+a'}-Y^{b'}Z^c, \ \ \text{and} \ \ h=-X^aZ^{c'}+Y^{b'}(Y^b+V),
$$ 
we can consider the exact sequence 
$$
\xymatrix{
0 \ar[r] & T^2 \ar[r]^{^t\mathbb{M}} & T^3 \ar[rr]^{\scriptsize \begin{pmatrix}
f & g & h
\end{pmatrix}} && T \ar[r] & R \ar[r] & 0
}
$$
of $T$-modules. By taking the $T$-dual, we get the presentation of $\omega_R$ of the form $T^3 \overset{\mathbb{M}}{\to} T^2 \to \omega_R \to 0$. 
Therefore, the complex of $R$-modules
$$
\xymatrix{
R^3 \ar[r]^{\mathbb{M}} & R^2 \ar[rr]^{\left(\begin{smallmatrix} Y^{b'} & -X^a \end{smallmatrix}\right) \quad \ } & & (x^a,y^{b'})R \ar[r] & 0
}
$$
 induces a natural epimorphism 
$$
\varphi: \omega_{R}\twoheadrightarrow (x^a, y^{b'})R
$$
of $R$-modules. 
Moreover, $\varphi$ is an isomorphism because $\omega_{R}$ is a torsionfree $R$-module of rank one and $x^a$ is a non-zerodivisor on $R$. Hence $\omega_{R}  \cong (x^a, y^{b'})R$, as claimed. 
We similarly have $$\omega_{R}  \cong (y^b+v, z^{c'})R\cong (z^c, x^{a'})R.$$ 
We also note that the isomorphisms can be obtained by using the procedure of \cite[Section 6.1.2]{V2}.
In particular, $R$ is not a Gorenstein ring, since the type of $R$ is two.

Next, we show that $R$ is a quasi-normal ring. 
Let $\p \in \Spec R$ with $\height_R \p \le 1$. 
If $x \not\in \p$, then $[\omega_{R}]_{\fkp}\cong (x^a,y^{b'})R_\p = R_\p$, so that $R_\p$ is a Gorenstein ring. 
Assume that $x\in \p$. Similarly, we may assume that $y,z \in \p$. 
Then, $v \notin \p$, since $\height_R \p \le 1$. 
Therefore, $[\omega_{R}]_{\fkp} \cong (y^b+v,z^{c'})R_\p = R_\p$, so that  $R_\p$ is a Gorenstein ring. Hence $R$ is a quasi-normal ring.

Finally, we prove that $R$ is a normal ring if and only if $a'=b'=c=1$. 
Assume $a'=b'=c=1$ and consider the ideal 
$$
J=\mathbf{I}_2\begin{pmatrix} Z^{c'} & (a+1)X^a & Y^b+V\\ -(b+1)Y^b + V & -Z & bXY^{b-1}\\ c'X^aZ^{c'-1} & -Y & -(c'+1)Z^{c'}\\ -Y & 0 & X \end{pmatrix}.
$$
Then $J+I / I$ is the Jacobian ideal of $R$ over $k$. A direct computation shows that $\sqrt{J+I} =(X,Y,Z,V)$ and hence, by the Jacobian criterion, the local ring $R_\fkp$ is regular for every $\fkp \in \Spec R \setminus\{(x,y,z,v)\}$. Hence $R$ is a normal ring.
Conversely, we assume $a' \ge 2$, or $b' \ge 2$, or $c \ge 2$. By taking
$$P = \begin{cases}
	(X,Y^b+V, Z) & \text{(if $c \ge 2$)}\\
	(X,Y, Z) & \text{(if $c =1$)},
 \end{cases}
$$
we then have $J \subseteq P$. 
We set $\p=PR$. Then $\height_R \p = 1$, but $R_\p$ is not a DVR. 
Indeed, because $\varepsilon = Y$ or $\varepsilon' = Y^b+V$ is invertible in $T_P$, we see that
$$JT_P = 
\begin{cases}
\left(Z^{c+c'} - X^{a'}(Y^b+V), \dfrac{X^{a+a'}}{\varepsilon^{b'}} - Z^c, - \dfrac{X^aZ^{c'}}{\varepsilon} + (Y^b+V) \right) \subseteq (Y^b+V) +(X,Z)^2 & \text{(if $c \ge 2$)}\\
\left(\dfrac{Z^{c+c'}}{\varepsilon'} - X^{a'}, X^{a+a'} - Y^{b'}Z^c, 
-\dfrac{ X^a Z^{c'}}{\varepsilon'} + Y^{b'} \right) \subseteq (Y) +(X,Z)^2 & \text{(if $c =1$)}
\end{cases}
$$ 
in $T_P$. Thus $R_\p= T_P/JT_P$ cannot be a DVR. Hence $R$ is not a normal ring.
As a special case, $R=k[X,Y,Z,V]/\mathbf{I}_2\left(\begin{smallmatrix}
X&Y+V&Z\\
Y&Z&X^2\\
\end{smallmatrix}\right)$ is a
quasi-normal ring, but not normal.
\end{Example}



\section{Reflexive canonical modules in dimension one}\label{one-dimensional}

Let $(A,\fkm)$ be a Cohen-Macaulay local ring with $\dim A = 1$ admitting the canonical module $\rmK_A$. In this section, we explore the question of when $A$ has a reflexive canonical module.  We denote by $\rmQ(A)$ the total ring of fractions of $A$. Throughout this section, we assume that there exists an $A$-submodule $K$ of $\rmQ(A)$ such that $A \subseteq K \subseteq \overline{A}$ and $K \cong \rmK_A$ as an $A$-module, where $\overline{A}$ denotes the integral closure of $A$ in $\rmQ(A)$. Note that the assumption is automatically satisfied if $\rmQ(A)$ is Gorenstein and the residue class field $A/\m$ is infinite; see \cite[Corollaries 2.8, 2.9]{GMT}.
For $A$-submodules $X$ and $Y$ of $\rmQ(A)$, let $X:Y = \{a \in \rmQ(A) \mid aY \subseteq X\}$. If we consider ideals $I, J$ of $A$, we set $I:_AJ =\{a \in A \mid aJ \subseteq I\}$; hence $I:_AJ = (I:J) \cap A$.


\begin{Proposition}\label{prop5.1}
The following conditions are equivalent{\rm \,:}
\begin{enumerate}[$(1)$]
\item $A$ is a Gorenstein ring{\rm \,;}
\item $K^2 : K = K${\rm \,;}
\item $\rmK_A$ is a reflexive $A$-module.
\end{enumerate}
\end{Proposition}

\begin{proof}
(1) $\Leftrightarrow$ (3) See Proposition \ref{1.6}.

(3) $\Leftrightarrow$ (2) Since $A:K=[K:K]:K= K:K^2$ (\cite[Bemerkung 2.5]{HK}), we have $$A:(A:K) = (K:K):(K:K^2) = [K:(K:K^2)]:K=K^2 : K.$$ Therefore, $K^2:K=K$ if and only if $A:(A:K)=K$, that is $\rmK_A$ is a reflexive $A$-module.
\end{proof}

Recall that an ideal $I$ of $A$ is called a {\it canonical ideal} of $A$, if $I \ne A$ and $I \cong \rmK_A$ as an $A$-module. By \cite[Corollary 2.8]{GMT}, there exists a canonical ideal $I$ of $A$. We then have the following.


\begin{Theorem}
Let $I$ be a canonical ideal of $A$. Then the following conditions are equivalent{\rm \,:}
\begin{enumerate}[$(1)$]
\item $A$ is a Gorenstein ring{\rm \,;}
\item $I^2 :_A I = I${\rm \,;}
\item $I/I^2$ is a free $A/I$-module{\rm \,;}
\item $I$ is a reflexive $A$-module.
\end{enumerate}
\end{Theorem}

\begin{proof}
By Proposition \ref{prop5.1}, it suffices to show $(2) \Rightarrow (1)$. Enlarging the residue class field $A/\fkm$ of $A$ if necessary, we may assume that $A/\m$ is infinite. Let $I = (x_1, x_2, \ldots, x_n)$~($n > 0$) so that each $(x_i)$ is a reduction of $I$. We set $K_i = x_i^{-1}I$ and choose a non-zerodivisor $b$ of $A$ so that $bK_i^2 \subseteq A$ for all $1 \le i \le n$. Let $J = bI$ and $y_i = bx_i$ for $1 \le i \le n$. Then $(y_i)$ is a reduction of $J$. Notice that $A/I$ and $A/J$ are both Gorenstein rings, since $I, J \cong \rmK_A$ as $A$-modules.

\begin{claim}\label{claim1}
$J^2 :_AJ= J$.
\end{claim}

\begin{proof}[Proof of Claim 1]
Suppose that $J^2:_A J \supsetneq J$. Then, $J :_A\m \subseteq J^2:_AJ$. Since $A/J$ is a Gorenstein ring, we have $J:_A\m = J+A\varphi$ for some $\varphi \in (J:_A\m) \setminus J$. Hence, $\frac{\varphi}{b} \in \rmQ(A)$ and $\m{\cdot}\frac{\varphi}{b} \subseteq I$, so that $\frac{\varphi}{b} \in I:\m$. Because $I \subsetneq I :_A \m \subseteq I:\m$ and $\ell_A\left((I:\m)/I\right)=1$ (since $A/I$ is a Gorenstein ring), we get $\frac{\varphi}{b} \in I:_A\m$, so that $\frac{\varphi}{b} \in A$. On the other hand, $\frac{\varphi}{b}{\cdot}I \subseteq I^2$, since $\varphi{\cdot}bI =\varphi J \subseteq J^2 = b^2I^2$. Consequently, $\frac{\varphi}{b} \in I^2:_AI = I$, whence $\varphi \in bI = J$, which is impossible. Thus $J^2:_AJ=J$.
\end{proof}

Let $\overline{y_i}$ denote the image of $y_i$ in $J/J^2$. We then have $J/J^2= \sum_{i=1}^n(A/J){\cdot}\overline{y_i}$, and therefore, $(0):_{A/J}\overline{y_i}=(0)$ for some $i$, since $A/J$ is a Gorenstein ring and $(0):_{A/J}J/J^2=(0)$ by Claim \ref{claim1}. Without loss of generality, assume $i=1$. Then, $J^2 :_A y_1 = J$. On the other hand, since $bK_1^2 \subseteq A$ and $K_1=y_1^{-1}J$, we get $b{\cdot}(y_1^{-1}J)^2 \subseteq A$ , whence $bJ^2 \subseteq (bx_1)^2$. Therefore, $J^2 \subseteq (bx_1^2) \subseteq (bx_1)=(y_1)$. Hence, $J^2 = y_1{\cdot}(J^2:_Ay_1)=y_1J$. Thus $A$ is a Gorenstein ring (see \cite[Theorem 3.7]{GMT}).
\end{proof}



\section{$q$-torsionfree canonical modules}\label{qTFCM}

The purpose of this section is to give a generalization of Proposition \ref{prop2} and Theorem \ref{refcano1}, which characterize local rings with $q$-torsionfree canonical modules for $q=1,2$. 

Let $R$ be a Noetherian (not necessarily local) ring and $q$ an integer. 
Let $M$ be a finitely generated $R$-module with a finite projective presentation ${\ds P_{1}  \stackrel{\sigma}{\rar} P_{0} \rar M \rar 0 }$. 
By applying the $R$-dual functor $(-)^{*}=\Hom_R(-, R)$, we obtain the exact sequence
\[ 0 \lrar M^{*} \lrar P_{0}^{*} \stackrel{\sigma^{*}}{\lrar} P_{1}^{*} \lrar \rmD_{\sigma}(M) \lrar 0\]
 of $R$-modules.
We set ${\ds \rmD(M)=\rmD_{\sigma}(M)}$ and call it the {\em Auslander transpose} of $M$.
Note that $\rmD(M)$ is uniquely determined up to projective equivalence. 

\begin{Definition}[{\cite[Definition 2.15]{AB}}]\label{def_q_torsionfree}
A finitely generated $R$-module $M$ is said to be {\em $q$-torsionfree} if $\Ext_{R}^{i}(\rmD(M), R) =0 $ for all $i=1, 2, \ldots, q$.
\end{Definition}

By \cite[Proposition 2.6]{AB}, there exists an exact sequence
\[ 0 \lrar \Ext^{1}_{R}(\rmD(M), R) \lrar M \stackrel{\varphi}{\lrar} M^{**} \lrar \Ext^{2}_{R}(\rmD(M), R) \rar 0\] of $R$-modules, where $\varphi$ is the canonical homomorphism, and furthermore, we have
$$
\Ext_{R}^{i+2}(\rmD(M), R) \cong \Ext_{R}^{i}(M^{*}, R) \ \ \ \text{for all} \ \ i >0.
$$
This shows $M$ is torsionless (resp. reflexive) if and only if $M$ is $1$-torsionfree ($2$-torsionfree). When $q \ge 3$, the $R$-module $M$ is $q$-torsionfree if and only if $M$ is reflexive and $\Ext_{R}^{i}(M^{*}, R)=(0)$ for all $i=1, 2, \ldots, q-2$.

\begin{Example}\label{ex3}
Let $S=k\lb X,Y, Z \rb$ be the formal power series ring over a field $k$ and set $A= S/[(X) \cap (Y, Z)]$. Let  $x, y, z$ denote the images of $X, Y, Z$ in $A$, respectively. Then we have ${\ds \Assh A=\left\{ (x) \right\} }$, ${\ds U=(x) }$, and ${\ds \rmK_{A}=A/U}$, where $U$ denotes the unmixed component of $(0)$ in $A$. 
By dualizing the exact sequence  
\[ A \stackrel{\cdot x}{\longrightarrow} A \rar A/(x)=\rmK_{A} \rar 0,\] we obtain 
\[ 0 \rar \rmK_{A}^{*} \rar A \stackrel{\cdot x}{\longrightarrow} A \rar A/(x) \rar 0.\] Thus, ${\ds \rmD(\rmK_{A}) = \rmK_{A}}$. Consider the free resolution of $\rmD(\rmK_{A})=D$:
\[  A^{5} \stackrel{\tau_{4}}{\longrightarrow} A^{3} \stackrel{\tau_{3}}{\longrightarrow} A^{2} \stackrel{\tau_{2}}{\longrightarrow}  A \stackrel{\cdot x}{\longrightarrow} A   \longrightarrow D   \longrightarrow 0,   \]
where ${\ds \tau_{2} = [y \; z], \; \tau_{3}=\left[ \begin{array}{rrr} x & 0 & z \\ 0 & x & -y  \end{array}  \right],\; \text{and}\; \tau_{4}= \left[ \begin{array}{ccccc} y & z & 0 & 0 & 0 \\ 0 & 0& y & z & 0 \\ 0 & 0 & 0 & 0& x    \end{array}  \right]. }$ 
Dualize this free resolution to obtain
\[ 0 \longrightarrow D^{*} \longrightarrow  A  \stackrel{\cdot x}{\longrightarrow} A \stackrel{\sigma_{2}}{\longrightarrow} A^{2} \stackrel{\sigma_{3}}{\longrightarrow} A^{3} \stackrel{\sigma_{4}}{\longrightarrow} A^{5},   \]
where ${\ds \sigma_{2}, \sigma_{3}, \sigma_{4}}$ are the transposes of ${\ds \tau_{2}, \tau_{3}, \tau_{4} }$, respectively.  Let ${\ds a \in \ker \sigma_{2} }$. Then 
\[  a \in \big[(0):_Ay\big] \cap \big[(0):_Az\big] = (x) \cap (x) =(x) = \Im(\cdot x).\] 
Thus, ${\ds \Ext^{1}_{A}(D, A) =(0)}$. Let 
${\ds 
\left(\begin{smallmatrix}
a_{1}\\
a_{2}\\
\end{smallmatrix}\right)
\in \ker \sigma_{3}}$.  
Then ${\ds a_{1}, a_{2} \in (0):_Ax=(y, z) }$ and ${\ds a_{1}z-a_{2}y =0}$. 
Hence ${\ds
\left(\begin{smallmatrix}
-a_{2}\\
a_{1}\\
\end{smallmatrix}\right)
\in \ker \tau_{2} = \Im(\tau_{3}) }$. 
Let ${\ds -a_{2} = c_{1} x - c_{3}z }$, and ${\ds a_{1}= c_{2}x + c_{3}y  }$ for some ${\ds c_{1}, c_{2}, c_{3} \in A}$. 
Then ${\ds c_{1} x = - a_{2} +c_{3}z \in (y, z) }$. 
Thus, $c_{1}=0$ and $a_{2}=c_{3}z$. Similarly, ${\ds a_{1}= c_{3}y }$. 
We obtain ${\ds 
\left(\begin{smallmatrix}
a_{1}\\
a_{2}\\
\end{smallmatrix}\right)
\in \Im \sigma_{2} }$. 
Then ${\ds \Ext^{2}_{A}(D, A)  =(0)}$. 
Hence $\rmK_{A}$ is $2$-torsionfree. 
Note that ${\ds \ker \sigma_{4}}$ is generated by 
$\left(\begin{smallmatrix}
x\\
0\\
0
\end{smallmatrix}\right)$, $\left(\begin{smallmatrix}
0\\
x\\
0\\
\end{smallmatrix}\right)$, 
$\left(\begin{smallmatrix}
0\\
0\\
y
\end{smallmatrix}\right)$,
$
\left(\begin{smallmatrix}
0\\
0\\
z
\end{smallmatrix}\right)$.
Then ${\ds \Ext^{3}_{A}(D, A)  \neq (0)}$. Thus, $\rmK_{A}$ is not $3$-torsionfree.
\end{Example}

\begin{Definition}[{\cite[Definition 2.15]{AB}}]
A finitely generated $R$-module $M$ is called {\em $q$-syzygy}, if there exist finite free $R$-modules $F_{1}, F_{2}, \ldots, F_{q}$ and an exact sequence ${\ds 0 \to M \to F_{1} \to F_{2} \to \cdots \to F_{q}}$ of $R$-modules.
\end{Definition}

Note that $(a)$ $M$ is torsionless if and only if $M$ is $1$-syzygy, (b) every $q$-torsionfree $R$-module is $q$-syzygy, and (c) if $M$ is $q$-syzygy and $x$ is an $R$-regular element, then $M/xM$ is $(q-1)$-syzygy as an $R/xR$-module.


Although the following theorem has been proved by Foxby in a more general setting involving Gorenstein modules, we restate it and give its proof in our context for the sake of completeness.
Recall that $R$ is $q$-Gorenstein if $R_{\fkp}$ is Gorenstein for every prime $\fkp$ with ${\ds \depth R_{\fkp} <q}$.

\begin{Theorem}[{\cite[Proposition 3.2]{F1}}]\label{3.1} 
Let $A$ be a Cohen-Macaulay local ring admitting the canonical module $\rmK_{A}$. Then the following conditions are equivalent{\rm \,:}
\begin{enumerate}[{\rm (1)}]
\item $A$ is $q$-Gorenstein{\rm \,;}
\item $\rmK_{A}$ is $q$-torsionfree{\rm \,;}
\item $\rmK_{A}$ is $q$-syzygy.
\end{enumerate}
\end{Theorem}

\begin{proof} Since $A$ is Cohen-Macaulay, we have ${\ds \Spec A=\Supp_A \rmK_{A} }$ and ${\ds [ \rmK_{A} ]_{\fkp} = \rmK_{A_{\fkp}} }$ for every $\fkp \in \Spec A$.  Notice that ${\ds \rmK_{A_{\fkp}} }$ is maximal Cohen-Macaulay as an $A_{\p}$-module.


$(1) \Rightarrow (2)$ Since every $A$-regular sequence is $\rmK_A$-regular,  the $A$-module $\rmK_{A}$ is $q$-torsionfree by \cite[Proposition 2.3]{F}.

$(2) \Rightarrow (3)$ This follows from \cite[Proposition 2.1]{F}.

$(3) \Rightarrow (1)$ Let $\fkp \in \Spec A$ with $\depth A_\fkp < q$.
Set $n=\dim A_\fkp$. When $n=0$, the ring $A_\fkp$ is Gorenstein. Assume $n>0$ and choose a system $f_1, f_2, \ldots, f_n$ of parameters of $A_\fkp$. Then it is an $A_\fkp$-regular sequence, so that $\rmK_{A_\fkp}/(f_1, f_2, \ldots, f_n)\rmK_{A_\fkp}$ is $1$-syzygy because $n<q$. 
Since $\rmK_{A_\fkp}/(f_1, f_2, \ldots, f_n)\rmK_{A_\fkp}\cong\K_{A_\fkp/(f_1, f_2, \ldots, f_n)A_\fkp}$, we conclude that $A_\fkp/(f_1, f_2, \ldots, f_n)A_\fkp$ is Gorenstein by Proposition \ref{prop2}, whence so is the ring $A_\fkp$. This completes the proof.
\end{proof}

As a direct consequence of Theorem \ref{3.1}, we have the following.

\begin{Corollary}\label{3.1.1}
Let $A$ be a Cohen-Macaulay local ring with $d=\dim A$ admitting the canonical module $\rmK_A$. 
Then $A$ is a Gorenstein ring if and only if $\rmK_{A}$ is $(d+1)$-torsionfree.
\end{Corollary}

Theorem \ref{3.1} and Corollary \ref{3.1.1} lead us to the question of the structure of a local ring $A$ with $q$-torsionfree canonical module without the assumption that $A$ is Cohen-Macaulay.


\vspace{0.3em}
For a subset $\Phi$ of prime ideals in a Noetherian ring $R$, we say that 
\begin{itemize}
\item $R$ satisfies $(\mathrm{S}_{n})$ on $\Phi$ if ${\ds \depth R_{\fkp} \geq \min \left\{n, \, \dim R_{\fkp} \right\} }$ for every $\fkp\in \Phi$,
\item $R$ satisfies $(\rmG_{n})$ on $\Phi$ if $R_{\fkp}$ is Gorenstein for every $\fkp\in \Phi$  with ${\ds \dim R_{\fkp} \leq n}$.
\end{itemize}

The main result of this section gives an answer to the above question.


\begin{Theorem} \label{p3.2}
Let $A$ be a Noetherian local ring admitting the canonical module $\rmK_A$. Suppose that $\rmK_{A}$ satisfies $(\mathrm{S}_{q})$. Then the following conditions are equivalent{\rm \,:}
\begin{enumerate}[\rm(1)]
\item $A$ satisfies both $(\mathrm{G}_{q-1})$ and $(\mathrm{S}_{q-1})$ on $\Supp_A \rmK_{A}${\rm \,;}
\item $\rmK_{A}$ is $q$-torsionfree{\rm \,;}
\item $\rmK_{A}$ is $q$-syzygy.
\end{enumerate}
\end{Theorem}

To show this, we need some auxiliaries. 
The following plays an important role in our argument.

\begin{Lemma}[{\cite[Lemma 4.9]{Dey-Takahashi}}]\label{l1.7}
Let $A$ be a Noetherian local ring and $M$ a nonzero finitely generated $A$-module. Assume that $q \ge \depth A+2$ and $M$ is $q$-syzygy. Then $\depth_A M=\depth A$.
\end{Lemma}

We apply Lemma \ref{l1.7} to get the following.

\begin{Theorem} \label{t1.8}
Let $R$ be a Noetherian ring and $M$ a finitely generated $R$-module.
If $M$ is $(q+1)$-syzygy, then one has 
$$
\depth R_{\fkp} \ge \min\{q, \depth_{R_\fkp} M_\fkp\} \ \ \ \text{for all} \ \ \fkp\in \Supp_R M.
$$
In particular, if $M$ satisfies $(\mathrm{S}_q)$, then $R$ satisfies $(\mathrm{S}_q)$ on $\Supp_R M$.
\end{Theorem}

\begin{proof} 
By localizing at $\fkp \in \Supp_R M$, it suffices to show $\depth R \geq \min\{q, \depth_R M \}$. If $ \depth R\ge q$, the assertion is obvious. Otherwise, if $\depth R< q$, the assertion follows from Lemma \ref{l1.7}.
\end{proof}

As consequences of Theorems \ref{3.1}, \ref{t1.8}, we get the following.

\begin{Corollary} \label{c3.41}
Let $A$ be a Noetherian local ring with $d=\dim A$ admitting the canonical module $\rmK_A$. Then the following conditions are equivalent{\rm \,:}
\begin{enumerate}[\rm(1)]
\item $A$ is Gorenstein{\rm \,;}
\item $\rmK_{A}$ is a $(d+1)$-torsionfree maximal Cohen-Macaulay $A$-module{\rm \,;}
\item $\rmK_{A}$ is a $(d+1)$-syzygy maximal Cohen-Macaulay $A$-module.
\end{enumerate}
\end{Corollary}

\begin{proof} We only need to show (3) $\Rar$ (1). By Theorem \ref{t1.8} we have that $\depth A=d$, so that $A$ is Cohen-Macaulay. Hence the assertion follows from Theorem \ref{3.1}.
\end{proof}

\begin{Corollary}\label{c3.4}
Let $(A, \fkm)$ be a Noetherian local ring with $d = \dim A$ admitting the canonical module $\rmK_A$. 
Furthermore, we assume one of the following conditions {\rm (i)} and {\rm (ii)}.
\begin{enumerate}[{\rm (i)}]
	\item $\rmH^i_\fkm(A)=(0)$ for every integer $i\not=0,1,d$.
	\item $d \le 2$.
\end{enumerate}
Then the following conditions are equivalent{\rm \,:}
\begin{enumerate}[\rm(1)]
\item $A$ is Gorenstein{\rm \,;}
\item $\rmK_{A}$ is $(d+1)$-torsionfree{\rm \,;}
\item $\rmK_{A}$ is $(d+1)$-syzygy.
\end{enumerate}
\end{Corollary}

\begin{proof}
(i) By passing to the $\fkm$-adic completion, we may assume $A$ is $\fkm$-adically complete. 
In view of \cite[(2.3) Satz]{Sch}, it follows that $\rmK_{A}$ is maximal Cohen-Macaulay. Therefore the assertion follows from Corollary \ref{c3.41}.

(ii) Since $\rmK_{A}$ satisfies $(\mathrm{S}_2)$, the assertion follows from Corollary \ref{c3.41}.
\end{proof}


\begin{Remark}
{\rm Let $A$ be a Noetherian local ring admitting the canonical module $\rmK_A$. We say that $A$ is {\it quasi-Gorenstein} if $\rmK_A \cong A$ as an $A$-module. When $d\ge 3$, there exist non-Gorenstein quasi-Gorenstein local rings of dimension $d$ (see e.g., \cite[Theorem 2.11]{A}). Notice that, in such a ring $A$, $\rmK_A$ is $q$-torsionfree for all $q \ge 1$. So, Corollary \ref{c3.4} fails without the condition {\rm (i)} or {\rm (ii)}.}
\end{Remark}

Based on the above observation, it is natural to raise the following question.

\begin{Question}\label{Q1}
Let $A$ be a Noetherian local ring with $d=\dim A\ge 3$ admitting the canonical module $\rmK_A$. When are the following conditions equivalent?
\begin{enumerate}[(i)]
\item $A$ is a quasi-Gorenstein ring, i.e., $\rmK_A \cong A$.
\item $\rmK_A$ is a $(d+1)$-torsionfree $A$-module.
\end{enumerate}
\end{Question}

In what follows, let $R$ be a Noetherian ring and $M$ a finitely generated $R$-module.
The equivalence of $(1)$ and $(2)$ in the next theorem was essentially proved by Auslander and Bridger \cite{AB}. Notice that this is a $q^{\th}$ version of \cite[Proposition 1.4.1]{BH}.

\begin{Theorem} \label{t1.4} 
The following conditions are equivalent{\rm \,:}
\begin{enumerate}[\rm(1)]
\item $M$ is $q$-torsionfree{\rm \,;}
\item $M$ satisfies the conditions below{\rm \,:}
\begin{enumerate}[\rm(i)]
\item $M_\fkp$ is $q$-torsionfree for every $\fkp\in \Supp_R M$ with $\depth R_\fkp < q${\rm \,;}
\item $M$ satisfies $(\widetilde{\rmS}_{q})${\rm \,;}
\end{enumerate}
\item $M$ satisfies the conditions below{\rm \,:}
\begin{enumerate}[\rm(i)]
\item $M_\fkp$ is $q$-torsionfree for every $\fkp\in \Supp_R M$ with $\depth_{R_\fkp} M_\fkp < q${\rm \,;}
\item $\depth_{R_\fkp} M_\fkp=\depth R_\fkp$ for every $\fkp\in \Supp_R M$ with $\depth R_\fkp < q-1$.
\end{enumerate}
\end{enumerate}
\end{Theorem}

\begin{proof}
Without loss of generality, we may assume $q \ge 1$.

$(1) \Rightarrow (2)$ This follows from \cite[Proposition 2.1]{F}.

$(2) \Rightarrow (3)$ (i) Let $\fkp\in \Supp_R M$ with $\depth_{R_\fkp} M_\fkp < q$. Since $M$ satisfies $(\widetilde{\rmS}_{q})$, we have $\depth_{R_\fkp} M_\fkp\ge \depth R_\fkp$. This implies $\depth R_\fkp<q$, and hence $M_\fkp$ is $q$-torsionfree. (ii) Let $\fkp\in \Supp_R M$ with $\depth R_\fkp < q-1$. Then $M_\fkp$ is $q$-torsionfree, so that $\depth_{R_\fkp} M_\fkp=\depth R_\fkp$ by Lemma \ref{l1.7}.

$(3) \Rightarrow (1)$ 
For each $i\in\{1, 2,\dots,q\}$, we set $E^i = \Ext^i_R(\rmD(M),R)$.
Suppose $E^q\not=0$ and seek a contradiction. Take $\fkp\in \Ass_R E^q$. Since $\fkp \in \Supp_R M$, by (i) we have $\depth_{R_\fkp} M_\fkp\ge q$.
Then by (ii), $\depth R_\fkp \ge q-1$.
By passing to the localization $R_{\p}$ at $\fkp$, we may assume $R$ is a local ring, $\depth R\ge q-1$, $\depth_R M\ge q$, and $\depth_R E^q=0$.

We proceed by induction on $q$.
First, assume that $q=1$.
Since $E^1$ is isomorphic to a submodule of $M$, it follows that $\depth_R M=0$, a contradiction. Thus $E^1=(0)$.
Next, we assume $q= 2$. 
Applying the depth lemma to the exact sequence $0\to M\to M^{**}\to E^2\to 0$ of $R$-modules, we get $\depth_R M=1$, as $\depth_R M^{**}\ge 1$. This is impossible, whence $E^2=(0)$.
Suppose $q \ge 3$ and the assertion holds for $q-1$, i.e., $M$ is $(q-1)$-torsionfree. 
Hence $E^1=\cdots=E^{q-1}=(0)$.
Consider a free resolution $(F_i, \partial_i)$ of $M^{*}$. 
Applying the $R$-dual functor $(-)^*$, we get the exact sequence 
\[
0 \to M^{**} \to F_0^* \xrightarrow[]{\partial_1^*} F_1^* \to \cdots \to F_{q-3}^* \xrightarrow[]{\partial_{q-2}^*} F_{q-2}^*
\]
of $R$-modules because $E^3=\cdots=E^{q-1}=(0)$.
Let $C$ be the cokernel of $\partial_{q-2}^*$.
Since $M$ is reflexive as an $R$-module, we obtain the exact sequence of the form
\[
0 \to M \to F_0^* \to \cdots \to F_{q-2}^* \to C \to 0.
\]
Since $E^q=\Ext^{q-2}_R(M^*,R)$ may be regarded as a submodule of $C$, we see that $\depth_R C=0$. Hence $\depth_R M=q-1$. This gives a contradiction.
Hence we conclude that $E^q=(0)$, which shows $M$ is $q$-torsionfree.
\end{proof}

\begin{Corollary} \label{t1.5}
Suppose that the following conditions are satisfied{\rm \,:}
\begin{enumerate}[\rm (a)]
\item $M_\fkp$ is $q$-torsionfree for every $\fkp\in\Supp_R M$ with $\dim R_\fkp<q${\rm \,;}
\item $M$ satisfies $(\mathrm{S}_q)${\rm \,;}
\item $\depth R_\fkp \ge \min\{q-1,\dim R_\fkp-1\}$ for every $\fkp\in \Supp_R M$.
\end{enumerate}
Then $M$ is $q$-torsionfree.
\end{Corollary}

\begin{proof}
We will check condition (3) in Theorem \ref{t1.4}. Let $\fkp\in \Supp_R M$. 
(i) Assume that $\depth_{R_\fkp} M_\fkp<q$. By (b), $\depth_{R_\fkp} M_\fkp=\dim R_\fkp$, so that $\dim R_\fkp<q$. Therefore $M_\fkp$ is $q$-torsionfree by (a). (ii) Assume that $\depth R_\fkp<q-1$. By (c), $\depth R_\fkp \ge \dim R_\fkp-1$, so that $\dim R_\fkp<q$. Therefore $M_\fkp$ is $q$-torsionfree by (a), whence $\depth_{R_\fkp} M_\fkp=\depth R_\fkp$ by Lemma \ref{l1.7}.
\end{proof}

\begin{Corollary} \label{c1.6}
Suppose that the following conditions are satisfied{\rm \,:}
\begin{enumerate}[\rm (a)]
\item $M$ satisfies $(\mathrm{S}_q)${\rm \,;}
\item $R$ satisfies both $(\mathrm{G}_{q-1})$ and $(\mathrm{S}_{q-1})$ on $\Supp_R M$.
\end{enumerate}
Then $M$ is $q$-torsionfree.
\end{Corollary}

\begin{proof}
By Corollary \ref{t1.5}, it suffices to show that $M_\fkp$ is $q$-torsionfree on $\{\fkp\in\Supp_R M \mid \dim R_\fkp<q\}$.
Let $\fkp\in \Supp_R M$ with $\dim R_\fkp <q$.
Then by (b), $R_\fkp$ is Gorenstein.
As $M$ satisfies $(\mathrm{S}_q)$, $\depth_{R_\fkp} M_\fkp \ge\dim R_\fkp$.
Hence $M_\fkp$ is maximal Cohen-Macaulay as an $R_\fkp$-module.
In particular, $M_\fkp$ is $q$-torsionfree.
\end{proof}

We are now ready to prove Theorem \ref{p3.2}.

\begin{proof}[Proof of Theorem \ref{p3.2}]



$(1) \Rightarrow (2)$ This follows from Corollary \ref{c1.6}.

$(2) \Rightarrow (3)$ This follows from \cite[Proposition 2.1]{F}.

$ (3) \Rightarrow (1)$ By Theorem \ref{t1.8}, the ring $A$ satisfies $(\mathrm{S}_{q-1})$ on $\Supp_A \rmK_A$. Let $\fkp\in \Supp_A \rmK_A$ with $\dim A_\fkp \le q-1$. Then $A_\fkp$ is a Cohen-Macaulay ring of dimension at most $q-1$. Hence Theorem \ref{3.1} implies that $A_\fkp$ is a Gorenstein ring.
\end{proof}

As consequences of Theorem \ref{p3.2}, we get the following corollaries.

\begin{Corollary} \label{p3.5}
Let $A$ be a Noetherian local ring admitting the canonical module $\rmK_{A}$.
Suppose that $q\ge 2$ and $\rmK_A$ satisfies $(\mathrm{S}_q)$. Consider the following conditions{\rm \,:}
\begin{enumerate}[\rm (1)]
\item $\rmK_A$ is $(q+1)$-torsionfree{\rm \,;}
\item $\rmK_A$ is $(q+1)$-syzygy{\rm \,;}
\item $A$ satisfies both $(\mathrm{S}_q)$ and $(\mathrm{G}_{q-1})$ on $\Supp_A \rmK_A${\rm \,;}
\item $A$ satisfies both $(\mathrm{S}_q)$ and $(\mathrm{G}_{q-1})$, that is, $A$ is $q$-Gorenstein{\rm \,;}
\item $A$ satisfies $(\mathrm{S}_q)$ and $\rmK_A$ is $q$-torsionfree{\rm \,;}
\item $A$ satisfies $(\mathrm{S}_q)$ and $\rmK_A$ is $q$-syzygy.
\end{enumerate}
Then the implications {\rm (1)}$\Leftrightarrow${\rm (2)}$\Rightarrow${\rm (3)}$\Leftrightarrow${\rm (4)}$\Leftrightarrow${\rm (5)}$\Leftrightarrow${\rm (6)} hold true.
\end{Corollary}

\begin{proof}
The implications $(1) \Rightarrow (2)$ and $(4) \Rightarrow (3)$ are clear.
The equivalence of $(3)$, $(5)$, and $(6)$ immediately follows from Theorem \ref{p3.2}. Thus it suffices to check the implications $(2) \Rightarrow (3)$, $(3) \Rightarrow (4)$, and $(2) \Rightarrow (1)$.

$(2) \Rightarrow (3)$ Theorem \ref{p3.2} shows the ring $A$ satisfies $(\mathrm{G}_{q-1})$ on $\Supp_A \rmK_A$. On the other hand, by Theorem \ref{t1.8}, we deduce that $A$ satisfies $(\mathrm{S}_q)$ on $\Supp_A \rmK_A$.

$(3) \Rightarrow (4)$ Since $q\ge 2$, by \cite[Lemma 1.1]{AG} we have $\Supp_A \rmK_A=\Spec A$. 

$(2) \Rightarrow (1)$ The implication $(2) \Rightarrow (4)$ guarantees that $A$ is $q$-Gorenstein. Hence $\rmK_A$ is $q$-torsionfree by \cite[Proposition 4.21]{AB}.
\end{proof}

Since $\rmK_A$ satisfies $(\mathrm{S}_2)$, from Corollary \ref{p3.5} we have the following. 


\begin{Corollary}
Let $A$ be a Noetherian local ring admitting the canonical module $\rmK_{A}$.
Consider the following conditions{\rm \,:}
\begin{enumerate}[\rm (1)]
\item $\rmK_A$ is $3$-torsionfree{\rm \,;}
\item $\rmK_A$ is $3$-syzygy{\rm \,;}
\item $A$ satisfies both $(\mathrm{S}_2)$ and $(\mathrm{G}_1)$, that is, $A$ is quasi-normal{\rm \,;}
\item $A$ satisfies $(\mathrm{S}_2)$ and $\rmK_A$ is $2$-torsionfree{\rm \,;}
\item $A$ satisfies $(\mathrm{S}_2)$ and $\rmK_A$ is $2$-syzygy.
\end{enumerate}
Then the implications {\rm (1)}$\Leftrightarrow${\rm (2)}$\Rightarrow${\rm (3)}$\Leftrightarrow${\rm (4)}$\Leftrightarrow${\rm (5)} hold true.
\end{Corollary}

\begin{Corollary}
Let $A$ be a Noetherian local ring admitting the canonical module $\rmK_{A}$.
Suppose that $q\ge 2$ and $\rmK_A$ satisfies $(\mathrm{S}_{q+1})$.
Then the following conditions are equivalent{\rm \,:}
\begin{enumerate}[\rm (1)]
\item $\rmK_A$ is $(q+1)$-torsionfree{\rm \,;}
\item $\rmK_A$ is $(q+1)$-syzygy{\rm \,;}
\item $A$ satisfies both $(\mathrm{S}_q)$ and $(\mathrm{G}_q)$ on $\Supp_A \rmK_A${\rm \,;}
\item $A$ satisfies both $(\mathrm{S}_q)$ and $(\mathrm{G}_q)$.
\end{enumerate}
\end{Corollary}

\begin{proof}
This follows from Theorem \ref{p3.2} and the fact that $\Supp_A \rmK_A=\Spec A$ (\cite[Lemma 1.1]{AG}). 
\end{proof}

\begin{Corollary}
Let $A$ be a Noetherian local ring with $d = \dim A$ which is a homomorphic image of a Gorenstein ring. Suppose that $q\ge \frac{d}{2}+1$ and $\rmK_A$ is $(q+1)$-syzygy satisfying $(\mathrm{S}_q)$. Then $A$ is a Cohen-Macaulay ring.
\end{Corollary}

\begin{proof}
By Theorem \ref{t1.8}, we see that $A$ satisfies $(\mathrm{S}_q)$. We may assume $d>0$. Then $q\ge 2$, so that $A$ is equidimensional by \cite[Lemma 1.1]{AG}. Furthermore, either $A$ is Cohen-Macaulay or $\depth A\ge q$.
We assume $\depth A\ge q$.
Since $\rmK_A$ satisfies $(\mathrm{S}_q)$, 
every $A$-regular sequence of length at most $q$ is $\rmK_A$-regular (\cite[Proposition 2.1]{F}).
The assertion follows from \cite[Corollary (2.6)]{FFGR} (see also \cite[Proposition 4.2]{F1}).
\end{proof}


\section{Examples of $q$-Gorenstein rings}

Closing this paper, in order to illustrate our theorems, we provide additional examples of Cohen-Macaulay and $q$-Gorenstein rings, i.e., rings with $(\rmS_q)$ and $(\rmG_{q-1})$ conditions, or equivalently, rings with $(\widetilde{\rmG}_{q-1})$ condition. 

\begin{Theorem}\label{5.1}
Let $A$ be a Gorenstein local ring with $d= \dim A \ge 3$ and let $a_1, a_2, \ldots, a_d$ be a system of parameters of $A$. Let $\fka = (a_1, a_2, \ldots, a_\ell)$~$(3 \le \ell \le d)$ and let $$\calR = A[a_1t, a_2t, \ldots, a_\ell t] \  \subseteq \ A[t]$$ be the Rees algebra of $\fka$, where $t$ denotes an indeterminate. Then,  $\calR$ is not a Gorenstein ring, but it is a Cohen-Macaulay $(\ell + 1)$-Gorenstein ring of dimension $d+1$.
\end{Theorem}

\begin{proof}
Recall that $\calR$ is a Cohen-Macaulay ring of dimension $d+1$. Let $S = A[X_1, X_2, \ldots, X_\ell]$ be the polynomial ring over $A$ and let $\varphi : S \to \calR$ denote the surjective homomorphism of $A$-algebras defined by $\varphi(X_i) = a_it$ for each $1 \le i \le \ell$. The homomorphism $\varphi$ preserves the grading and ${\rm Ker}( \varphi) = \mathbf{I}_2\left(\begin{smallmatrix}
X_1&X_2&\ldots&X_\ell\\
a_1&a_2&\ldots&a_\ell\\
\end{smallmatrix}\right)$
is the perfect ideal of $S$ of grade $\ell - 1$ generated by the $2 \times 2$ minors of the matrix $\left(\begin{smallmatrix}
X_1&X_2&\ldots&X_\ell\\
a_1&a_2&\ldots&a_\ell\\
\end{smallmatrix}\right)$. We set $I ={\rm Ker}(\varphi)$. We then have the following.

\begin{claim}
Let $P \in \Spec S$ such that $I \subseteq P$ but $(X_1, X_2, \ldots, X_\ell)+(a_1, a_2, \ldots, a_\ell) \not\subseteq P$. Then, $S_P/IS_P$ is a Gorenstein ring.
\end{claim}

\begin{proof}[Proof of Claim 2]
We may assume that $X_1 \not\in P$. Let $\widetilde{S}=S[\frac{1}{X_1}]$, $\widetilde{A}=A[X_1, \frac{1}{X_1}]$, and $Y_i= \frac{X_i}{X_1}$ for $2 \le i \le \ell$. Then, $\widetilde{S} = \widetilde{A}[Y_2, Y_3, \ldots, Y_\ell]$ and $I\widetilde{S}= (a_i - a_1Y_i \mid 2 \le i \le \ell)\widetilde{S}$. Because $a_1 \widetilde{S}+(a_i - a_1Y_i \mid 2 \le i \le \ell)\widetilde{S} =(a_i \mid 1 \le i \le \ell)\widetilde{S}$ and $a_1, a_2, \ldots, a_\ell$ is an $\widetilde{S}$-regular sequence, the sequence $a_2-a_1Y_2, a_3 - a_1Y_3, \ldots, a_\ell - a_1Y_\ell$ is $\widetilde{S}_P$-regular, so that $S_P/IS_P$ is a Gorenstein ring.
\end{proof}

Let $P \in \Spec S$ and suppose that $I \subseteq P$. We set $\fkp = \varphi(P) \in \Spec \calR$. Then, $(X_1, X_2, \ldots, X_\ell) + (a_1, a_2, \ldots, a_\ell) \not\subseteq P$ if $\h_SP < 2\ell$, while 
$$\h_\calR \fkp = \h_{S/I}P/I= \h_SP - (\ell - 1).$$ Therefore, if $\h_\calR\fkp < \ell +1$, then $\h_SP - (\ell - 1) < \ell + 1$, that is $\h_SP < 2\ell$, so that $(X_1, X_2, \ldots, X_\ell) +(a_1, a_2, \ldots, a_\ell) \not\subseteq P$, whence $\calR_\fkp = S_P/IS_P$ is a Gorenstein ring by Claim 2. Thus, $\calR$ is an $(\ell+1)$-Gorenstein ring.
\end{proof}

Since the proofs of the following assertions are standard, we left them to the interested readers.

\begin{Lemma}
Let $\varphi : A \to B$ be a flat local homomorphism of Noetherian local rings and $q \ge 1$ be an integer. Then the following conditions are equivalent{\rm \,:}
\begin{enumerate}[$(1)$]
\item $B$ is a $q$-Gorenstein ring{\rm \,;}
\item $A$ is a $q$-Gorenstein ring and $B_P/\fkp B_P$ is a Gorenstein ring for every $P \in \Spec B$ with $\depth B_P < q$, where $\fkp = \varphi^{-1}(P)$.
\end{enumerate}
\end{Lemma}

\begin{Proposition}\label{5.4} 
Let $R$ be a Noetherian ring. Then the following assertions hold true.
\begin{enumerate}[$(1)$]
\item Let $q \ge 1$ be an integer. Then $R[t]$ is a $q$-Gorenstein ring if and only if $R$ is a $q$-Gorenstein ring, where $t$ is an indeterminate.
\item Let $H$ be a symmetric numerical semigroup. If $R$ is a $q$-Gorenstein ring, then the semigroup ring $R[H]$ of $H$ over $R$ is a $q$-Gorenstein ring. 
\item Let $X = \{X_{ij}\}_{1 \le i \le \ell, 1 \le j \le m}$ be indeterminates where $\ell, m \ge 2$, and set $T=R[X]$.  Let $t$ be an integer such that $2 \le t \le \min~\{\ell, m \}$ and let $I =\mathbf{I}_t(X)$ denote the ideal of $S$ generated by the $t \times t$ minors of the matrix $X$. We set $S =T/I$.
\begin{enumerate}
\item[$(\rm{a})$] Let $\ell = m$. If $R$ is a $q$-Gorenstein ring, then $S$ is a $q$-Gorenstein ring. 
\item[$(\rm{b})$] Suppose that $R$ is a field and let $t=2$. Then $S$ is a $d$-Gorenstein ring, where $d =\ell + m -1$.
\end{enumerate}
\end{enumerate}
\end{Proposition}


\end{document}